\documentclass[12pt]{amsart}
	\usepackage{amsthm}
	\usepackage[utf8]{inputenc}
	\usepackage{appendix}

	\usepackage{cite}
	\usepackage{amssymb,amsmath,amsthm}
	\usepackage[dvipsnames]{xcolor}
	\usepackage{tikz}
	\usepackage{pgfplots}
	\pgfplotsset{compat=1.16} 
	\usetikzlibrary{patterns}
	\usepgfplotslibrary{fillbetween}
	
	\newcommand{\amsprimary}[1]{{\footnotesize\noindent AMS 2010 \textit{Mathematics subject
	classification:} Primary #1\vspace{1pc}}}
	\newcommand{\keywordsnames}[1]{{\footnotesize\noindent\textit{Key words:} #1\vspace{1pc}}}

	\newtheorem{theorem}{Theorem}

	\newtheorem{teo}{Theorem}
	\newtheorem{prop}{Proposition}

	\newtheorem{lemma}{Lemma}
	\theoremstyle{definition}
	\newtheorem{de}{Definition}

	\newtheorem{rmk}{{\bf Remark}}
	
	\title{An Iteration method for solving elliptic equations}

	\author{J. C. Cortissoz}
	\address[J. C. CORTISSOZ]{Departmento de Matemáticas\\
	  Universidad de Los Andes\\
	  Bogotá DC, Colombia}
	\email{jcortiss@uniandes.edu.co}
	
	\author{J. Torres Orozco}
	\address[J. TORRES OROZCO]{Facultad de Ciencias\\
	  Universidad Nacional Autónoma de México\\
	  Ciudad de México, México}
	\email{jonatan.tto@gmail.com}

	\date{}
	
\begin{document}
	
	\maketitle
	
	\begin{abstract}
	In this paper we use a natural iteration technique to
	prove
	existence of solutions to nonlinear Dirichlet problems. Among
	the examples included is the prescribed mean curvature equation.
	The nature of the technique allows applications to unbounded domains,
	as we show with domains of the form
	$\mathbb{R}^{n-1}\times\left(-d/2,d/2\right)$.
	\end{abstract}
	 
	{\keywordsnames {Nonlinear elliptic equations, iteration methods, prescribed mean curvature problem.}}
	
	{\amsprimary {35J60, 53A10.}}
	\section{Introduction}
\label{sec: Int}
For any elliptic operator $L$ on $C^{2, \alpha}(\Omega)$ functions, for $0<\alpha<1$, on a smooth domain $\Omega$ in $\mathbb{R}^n$ 
(though, being $C^{2,\alpha}$ is enough)
we can consider the Dirichlet problem
\begin{equation*}
    \begin{cases}    
    L[u](x)=f(x, u(x), \nabla u(x))& x\in \Omega,\\
    u(x)=0 & x\in\partial\Omega.
    \end{cases}
    \end{equation*}
In this paper we shall study the case when $L=-\Delta$, and show some conditions on $f$ that ensure existence of solutions to this problem
(uniqueness will hold in a ball of certain radius in $C^{2,\alpha}\left(\Omega\right)$).
\newline

The starting point of our strategy is the following natural idea.
We construct $\{u_i\}$ a sequence of functions in  the Banach space $C^{2, \alpha}(\Omega)$, 
starting with $u_0=0$, and then, once $u_i$ has been defined,
defining $u_{i+1}$ as the solution to the Dirichlet problem:
\begin{equation}
\label{DP: general}
    \begin{cases}
    L[u_{i+1}](x)=f(x, u_i(x), \nabla u_i (x))& x\in \Omega,\\
    u_{i+1}(x)=0 & x\in\partial\Omega.
    \end{cases}
    \end{equation}
The classical boundary Schauder estimates give an upper bound for the  $C^{2, \alpha}(\Omega)$ norm of the $i+1$-term:
\begin{equation*}
\label{ineq: Schauder}
    \left|u_{i+1}\right|_{2, \alpha}\leq \Lambda(n, \alpha, \lambda, \Gamma, \Omega) \left|f(\cdot, u_i(\cdot),
    \nabla u_i\left(\cdot\right))\right|_{\alpha}.
\end{equation*}
Here $\lambda$ is the ellipticity bound, while $\Gamma$ only depends on the coefficients of the differential operator $L$. We will refer to $\Lambda$ as the {\em Schauder constant}.
\newline

Therefore our objective translates into finding suitable conditions on $f$ such that 
the sequence $\{u_{i}\}$ converges to a classical solutions of the original boundary problem.
\newline

In order to motivate the strategy that we are to follow, let
us give a short explanation on the method with as little 
technicalities as possible, showing in passing the 
scope of our results and what assumptions are needed for 
it to work. Our scheme is as follows:
\medskip

\begin{itemize}
\item[1.] Consider $\{u_i\}$ a sequence of functions in $C^{2, \alpha}(\Omega)$ defined through (\ref{DP: general}), starting at $u_0=0$. For instance,
one of the conditions that we impose on $f$ is that
it can be written as $h\left(x\right)+F\left(\nabla u\right)$, and that $F$
satisfies some sort of Lipschitz condition. See Theorem \ref{thmA}.
\medskip

\item[2.] Then, we use the Schauder estimates,
together with a boundedness condition or the existence of a fixed point for
a certain function constructed out of $f$,
to obtain a uniform bound of the $C^{2, \alpha}$ norms $\{|u_i|_{2, \alpha}\}$.
Thus for every $0<\beta<\alpha$, any subsequence has a convergent subsequence. The uniform bound 
will depend on the Schauder constant, the corresponding Lipschitz constant of $F$, 
and the H\"older norm of $h$.
\medskip

\item[3.] For any $i\geq 1$, set $v_i=u_{i+1}-u_i$, use integration by parts formula and the Sobolev embedding $
H_0^{1, 2}(\Omega)\hookrightarrow L^2(\Omega)$ to bound:
\[
\left\|v_{i+1}\right\|_{1, 2}\leq \rho\left\|v_i\right\|_{1, 2}.
\]
Such $\rho$ depends on $F$ and the domain. More precisely, its dependence on the domain is given by its measure, if it is bounded, or by its slab diameter, if it is one direction bounded (see Definition \ref{def:slab_diameter}). By having $\rho<1$ one ensures that all the subsequential limits 
of the sequence $\{u_i\}$ are the same, and therefore the original sequence is convergent in $C^{2, \beta}(\Omega)$
for all $0<\beta<\alpha$, and that its limit is a solution to the original problem.

\medskip
\item[4. ]
The fact that the Schauder constant depends on the local geometry of the boundary, and the existence
of $\eta>0$ such that $\partial \Omega$ has a smooth collar neighborhood of width $\eta$, plus the fact that
there is a version of the Poincar\'e inequality where the constant only depends on the slab diameter of the domain, 
allows us
to extend our results to certain unbounded domains.

\medskip
\item[5. ] Our methods extend to nonhomogenous boundary conditions, and we show how this
can be done in Remark
\ref{rmk:nonhomogenous}. We have decided to present our methods with homogeneous boundary conditions
for the sake of exposition and simplicity of the statements of our main results.

\end{itemize}

\medskip

The iteration method used in this paper has been studied by the authors in other contexts,
mainly using a maximum principle approach in \cite{CortTorres2021A, CortTorres2021B}.
\newline

The main results are summarised in the following theorems. 

\begin{theorem}
\label{thmA}
Let 
\[
f= h\left(x\right)+F\left(\nabla u\right),
\]
where $h$ is a H\"older continuous function and $\Omega$
a bounded smooth domain.
Suppose that for any $u\in C^{2, \alpha}(\Omega)$, and for every $x, y\in \Omega$:
\[
\left|F(\nabla u \left(x\right))-F(\nabla u\left(y\right))\right|\leq K\left|\ | \nabla u|^m(x)-|\nabla u|^m(y) \right|,
\]
for some constant $K>0$, and a real number $m\geq 2$. Then,
for $K$ small enough, the Dirichlet problem:
\begin{equation}
\label{eqn: BVIge}
    \begin{cases}
    \triangle u(x)= f & x\in\Omega,\\
    u(x)=0 & x\in\partial\Omega
    \end{cases}
\end{equation}
has a solution in $C^{2, \alpha}(\Omega)$. 

\end{theorem}
\medskip

\begin{theorem}
\label{thmB}
Let
\[
f=\gamma\left(x\right)g\left(u\right)\left|\nabla u\right|^m+h\left(x\right),
\]
where $h, \gamma:\Omega\longrightarrow \mathbb{R}$ are H\"older continuous functions, with
$\Omega$ a bounded smooth domain. Assume that $g$ is continuous, nondecreasing
and that $g'\left(x\right)\leq x^k$. Then if either 
\begin{itemize}
\item[1. ]
$\left|\gamma\right|_{\alpha}$ is small enough,
\item[2. ]
or if $\delta<1$, $k$ is large enough and $\left|\gamma\right|_0$ is small enough,
\end{itemize}
then the Dirichlet problem (\ref{eqn: BVIge}) has a solution. 
\end{theorem}

\medskip
The type of nonlinearity in Theorem \ref{thmA} when $F$ is globally Lipschitz has been 
studied by Xu in \cite{Xu2020}, an the type of nonlinearity addressed   
in Theorem \ref{thmB} has been considered before by Porretta and Segura de Le\'on in \cite{Porretta2006}.

\medskip
As the careful reader might have noticed in the
statements 
of our results, there are no uniqueness claims: however, it will be clear 
from our proofs that the solutions are unique within a certain ball in $C^{2,\alpha}\left(\Omega\right)$.
Also, what might be interesting of the methods developed in this paper is that they can be carried
over to certain unbounded domains. Avoiding generality for the sake of clarity, we show
that Theorems \ref{thmA} and \ref{thmB} are valid in the case of a slab. 
So we have the following.

\medskip
\begin{theorem}
\label{thmC}
Theorems \ref{thmA} and \ref{thmB} are valid when
\[
\Omega=\mathbb{R}^{n-1}\times\left(-d/2,d/2\right).
\]
\end{theorem}
We shall present a proof of the previous theorem in the case $n=2$, as it is easier to visualize, but
the proof of the general case follows along the same lines and we have no doubt can be generalised by the 
interested reader.

\subsubsection*{The prescribed mean curvature problem}

The classical prescribed mean curvature problem for a domain $\Omega$ in $\mathbb{R}^n$, with boundary $\partial\Omega$ asks for the existence and uniqueness of solutions to the Dirichlet problem:
\begin{equation*}
\nabla\cdot \left(\frac{\nabla u}{\sqrt{1+\left|\nabla u\right|^2}}\right)=nH
\end{equation*}
with $u=\varphi$ on $\partial \Omega$. The function $H$ is the prescribed mean curvature of the graph of the solution $u:\Omega\to \mathbb{R}$. 

\medskip
The previous procedure needs to be adapted to study the prescribed mean curvature problem. The part that prevents it
this part to be a direct application of Theorem \ref{thmA} is that the function $F$ involves second partial derivatives. 
\newline

In the pursuit of solving the Dirichlet problem 
for the mean curvature equation, different conditions on the geometry and regularity of $\Omega$ and $H$
have been imposed. For bounded euclidean domains $\Omega$ of class $C^{2, \alpha}$ and a function $\varphi$ in $C^{2, \alpha}(\bar{\Omega})$, for $0<\alpha<1$, the classical theorem by J. Serrin establishes that a neccesary and sufficient condition to ensure the solvability is that the mean curvature of the boundary surface is bounded above by $\frac{n}{n-1} |H|_0$ (see \cite{Serrin1969}). His proof relies on 
the maximum principle and estimates of the gradient of the solution at the domain and its boundary. 
By assuming that  $\Omega$ is a convex domain contained in an annulus of small diameter, M. Bergner was able to obtain a Serrin's type theorem, where $H$ is a convex $C^{1, \alpha}$ function of small $C^0$-norm \cite{Bergner}. D. Gilbarg and N.S. Trudinger proved a similar result by requiring an upper bound of the domain and an upper bound of $|H|_0$, in terms of the Poincaré-Sobolev constant. See the {\it apriori} estimates of Theorem 10.10 and condition 10.32 in \cite{GilTru1983}, for an elliptic operator in a divergence form. In a more general context, if the initial boundary $\varphi$ is a Lipschitz function with constant $0\leq K<1/(\sqrt{n-1})$, Bergner proved the existence of solutions  for the prescribed mean curvature problem. 
\newline

In Section \ref{Sec: GenExis} we obtain some general results for the appropriate convergence of the iterative method, but they cannot be applied directly to the prescribed mean curvature problem. Nevertheless, we will follow the same scheme of solution. Then we obtain analogues to Serrin's theorem on more general domains,
including some classes of 
unbounded domains. A similar approach was used by P. Amster and M. C. Mariani in \cite{AmsMar2003}, using the Schauder fixed point theorem. See also Theorem 11.8 in \cite{GilTru1983}. In \cite{Tsukamoto2020}, Y. Tsukamoto proved
Serrin's theorem using a variational approach. In every case, the conditions imposed on the domain are essential, for instance, Serrin's theorem fails in the absence of the convexity condition. 
\medskip

Our approach for proving the existence and uniqueness avoids imposing a convexity
assumption on $\Omega$. Here the main result that we present is:

\begin{theorem}
\label{thmE}
Let $\Omega$ be a smooth bounded domain or a slab.
There is an $\varepsilon\left(\Omega\right)>0$ such that if $\left|H\right|_{\alpha}< \varepsilon$,
then there is a solution to 
\[
\nabla\cdot \left(\frac{\nabla u}{\sqrt{1+\left|\nabla u\right|^2}}\right)=nH \quad \mbox{in}\quad \Omega,
\]
with $u|_{\partial \Omega}=0$. 
\end{theorem}
\medskip

Existence theorems for the prescribed mean curvature equation 
in unbounded domains are scarce. In fact, there are some results when the
domain is in the complement of a cone (so the slabs are not included, see \cite{Jin2007, JuLiu2016}), and
for the case of constant mean curvature in the slab (strip) when $n=2$ as in \cite{Lopez}. Also,
uniqueness can be claimed within a ball of a specific radius in $C^{2,\alpha}\left(\Omega\right)$,
and which is of the order of $\Lambda\left|H\right|_{\alpha}$, where $\Lambda$ is the
Schauder constant of the Laplacian in the domain being considered.
\medskip

\subsubsection*{The method for other elliptic operators}

We specialize the estimates for the Laplace operator $\triangle$, nevertheless the general results can be obtained for any elliptic operator $L:C^{2, \alpha}(\Omega)\to C^{2, \alpha}(\Omega)$. The idea behind follows the five guidelines that we sketched out. We also suggest to compare it with Theorem 10.10 and condition 10.32 of \cite{GilTru1983}. 
\newline

Let $L$ be a second order elliptic quasilinear operator defined on $C^{2, \alpha}(\Omega)$. Then it can be written as:
\[
Lu=\sum_{i, j=1}^n a^{ij} \partial^i\partial^ j u-B(x, u,  \nabla u),
\]
where $a^{ij}$ are essentially bounded  functions on $\Omega$, with $a^{ij}=a^{ji}$, and $B:\Omega\times \mathbb{R}\times \mathbb{R}^n\to \mathbb{R}$. The ellipticity condition gives the existence of a constant $\lambda>0$ such that 
\[
\sum_{i,j=1}^n a^{ij} \xi^i\xi^j\geq \lambda |\xi|^2
\] 
for all vectors $\xi=(\xi^1, \xi^2, \dots, \xi^n)$ in $\mathbb{R}^n$. 
\newline

Let $u$ in $C^{2, \alpha}(\Omega)$ be a solution of the Dirichlet problem:
\[
\begin{cases}
\sum_{i, j=1}^n a^{ij} \partial^i\partial^ j u=f(x, u, \nabla u)+B(x, u,  \nabla u) & x\in\Omega\\
u=0 & x\in\partial\Omega
\end{cases}
\]

In order to fulfill the method, first we need to impose a Lipschitz-type condition on $F:=f(x, u, \nabla u)+B(x, u,  \nabla u)$, as in Theorem \ref{thmA}. For the step 2, we may invoke the classical boundary Schauder estimates for an elliptic operator. The step 3 relies on a gradient estimate along a sequence given by (\ref{DP: general}). In fact, we have:
\[
\lambda \left\|\nabla u\right\|_2^2\leq \int_{\Omega}F\cdot u.
\]
Therefore, we are able to proceed as in Proposition \ref{prop: Sobound}. Finally, step 4 follows identically from a Poincaré inequality. 
\newline

Consider the particular case when $L$ is given in divergence form:
\[
Lu={\rm div} A(x, u, \nabla u)-B(x, u,  \nabla u)
\]
for a vector function $A:\Omega\times \mathbb{R}\times \mathbb{R}^n\to \mathbb{R}^n$.
\newline

%

For $F: \Omega\to \mathbb{R}$ a $C^{\alpha}(\Omega)$-function, let $u$ be a $C^{2, \alpha}(\Omega)$-solution of the Dirichlet problem:
\[
\begin{cases}
{\rm div} A(x, u, \nabla u)=f(x, u, \nabla u)+B(x, u,  \nabla u) & x\in\Omega\\
u=0 & x\in\partial\Omega
\end{cases}
\]

Note that in this case we also have an integration by parts formula, by using the divergence theorem. Since $u=0$ on $\partial \Omega$ this formula reads:
\[
-\int_{\Omega}u\ {\rm div A}= \left\|\nabla u\right\|_2^2= \int_{\Omega} F\cdot u.
\]
\medskip

\subsubsection*{Outline of the paper} In Section \ref{sec: NotaPre} set notation and give some basic results that will be used along the work. A proof of Theorems \ref{thmA} and \ref{thmB} is given in Section \ref{Sec: GenExis}. A proof for the 
existence result on the prescribed mean curvature problem is presented in \ref{sec: pmcp}. {Finally, in Section \ref{sec: Unbounded} we explain how these results work on strips, which in turn shows
how our arguments can be extended to some unbounded domains which are bounded in one direction.}
\medskip

\section{Notation and technical preliminaries}
\label{sec: NotaPre}

Consider a smooth domain $\Omega$ with boundary $\partial\Omega$. Denote by $C^{k, \alpha}(\Omega)$ the Hölder space on $\Omega$, formed by bounded functions in $\Omega$ with $k$ derivatives of finite $C^{ \alpha}$-norm. We 
employ the following notation.
\medskip

The $C^{\alpha}$ semi-norm is given by:
\[
\left[f\right]_{\alpha}=\sup_{x\neq y}\frac{\left|f\left(y\right)-f\left(x\right)\right|}{\left|y-x\right|^{\alpha}},
\]
the $C^{\alpha}$ norm is:
\[
\left|f\right|_{\alpha}=\left|f\right|_0+\left[f\right]_{\alpha},
\]
and the $C^{k, \alpha}$ norm is:
\[
|f|_{k, \alpha}=\sum_{|\beta|\leq k}|\partial^{\beta}f|_0+\sum_{|\beta|=k}\left[\partial^{ \beta}f\right]_{\alpha},
\]
for any multi-index $\beta=(\beta_1, \beta_2, \dots, \beta_s)$, where $|\beta|=\beta_1+ \beta_2+ \dots+ \beta_s$. 
\newline

The $L^p$-norm on the domain $\Omega$ is given by:
\[
\left\|f\right\|_p=\left(\int_{\Omega}f \right)^{1/p},
\]
with respect to the Lebesgue measure of $\mathbb{R}^n$. For a open set $\Omega$,
the Sobolev space
 $H_0^{1, p}(\Omega)$ is defined as the completion of the
 smooth functions $f:\Omega\longrightarrow \mathbb{R}$ whose support is contained in $\Omega$ 
 under the norm:
\[
\left\|f\right\|_{1, p}=\left(\int_{\Omega}\sum_{|\beta |\leq k}|D^{\beta}f|^p\right)^{1/p},
\]
where the integration is with respect to the Lebesgue measure of $\mathbb{R}^n$.
\newline

Recall that the Sobolev embedding theorem says that for $N>2$, the Banach space $H_0^{1, p}(\Omega)$ is continuously embedded into $L^q(\Omega)$, for any $q$ with $1\leq q\leq \frac{2N}{N-2}$, with the property that the embedding is also compact if $q<\frac{2N}{N-2}$. We also have the {\it Poincaré inequality} in its usual form:
\begin{equation}
\label{ineq: Poinc}
\left\|f\right\|_q\leq \left(\frac{|\Omega|}{\omega_n} \right)^{1/n}\left\|\nabla f\right\|_p
\end{equation}
for every $f\in H_0^{1, p}(\Omega)$ and $1\leq p<\infty$, where $|\Omega|$ is the Lebesgue measure of $\Omega$ and $\omega_n$ is the volume of the unit ball in $\mathbb{R}^n$. Below, we shall show a slab formulation of this inequality that
will be useful to extend our results to unbounded domains.
\newline

\begin{de}
\label{def:slab_diameter}
Let $\Omega$ be a Lipschitz domain in $\mathbb{R}^n$ that is bounded in one direction. The {\em slab} diameter $\delta:=\delta\left(\Omega\right)$ of $\Omega$ is defined as the infimum of the distances of two parallel 
hyperplanes such that $\Omega$ is totally contained in the region between them.
\end{de}

The following result is a version of the Poincaré inequality in terms of the slab diameter of a domain bounded in one direction. We will refer to it as the {\em slab Poincaré inequality}. It might be fairly well known, nevertheless we could not find it in the usual literature. For the sake of completeness, we include a proof. This inequality is 
a fundamental tool to extend some results on unbounded domains.
\medskip

\begin{lemma}[Slab Poincaré inequality]
For every $f\in H_0^{1, 2}(\Omega)$ we have the Poincaré inequality:
\begin{equation}
\label{ineq: PoincSlab}
\left\|f\right\|_2\leq \frac{\delta}{\sqrt{2}} \left\|\nabla f\right\|_2\leq \delta  \left\|\nabla f\right\|_2
\end{equation}
\end{lemma}
\begin{proof}
Without loss of generality we may assume that $\Omega=\mathbb{R}^{n-1}\times (0, \delta)$. For any one direction bounded domain the proof is similar, by means of the Coarea formula. 
\newline

Denote by $(\mathbf{x}, t)=(x_1, \dots, x_{n-1}, t)$ the coordinates of $\Omega$. Then by the Fundamental Theorem of Calculus and the Cauchy-Schwarz inequality, for any $0<t\leq\delta$, we can estimate
\begin{eqnarray*}
|f(\mathbf{x}, t)|=\left|\int_0^t \frac{\partial f}{\partial \tau}(\mathbf{x}, \tau)\ d\tau\right| &\leq & \left( \int_0^td\tau \right)^{1/2}\ \left(\int_0^t\left|\frac{\partial f}{\partial \tau}(\mathbf{x}, \tau)\right|^2 d\tau\right)^{1/2}\\
&\leq& t^{1/2}\ \left(\int_0^{\delta}\left|\frac{\partial f}{\partial \tau}(\mathbf{x}, \tau)\right|^2 d\tau\right)^{1/2}.
\end{eqnarray*}
Thus, integrating with respect to $t$,
\begin{eqnarray*}
\int_0^{\delta}f(\mathbf{x}, t)^2\ dt&\leq& \left(\int_0^{\delta}t\ dt\right)\int_0^{\delta}\left|\frac{\partial f}{\partial \tau}(\mathbf{x}, \tau)\right|^2 d\tau\\
&\leq&\frac{\delta^2}{2} \int_0^{\delta}\left|\frac{\partial f}{\partial \tau}(\mathbf{x}, \tau)\right|^2 d\tau.
\end{eqnarray*}
Therefore:
\begin{eqnarray*}
\left\|f\right\|_2^2= \int_{\mathbb{R}^{n-1}}\int_0^{\delta}f(\mathbf{x}, t)^2\ dt d\mathbf{x}&\leq&\frac{ \delta^2}{2} \int_{\mathbb{R}^{n-1}} \int_0^{\delta}\left|\frac{\partial f}{\partial \tau}(\mathbf{x}, \tau)\right|^2 d\tau d\mathbf{x}\\
\ \\
&\leq& \frac{\delta^2}{2}\left\|\nabla f\right\|_2^2.
\end{eqnarray*}\end{proof}
\medskip

\begin{lemma}
\label{lm: holderprop}
For any $f, g$ two $C^{\alpha}$-continuous functions on a Lipschitz domain $\Omega$ in $\mathbb{R}^n$ then the Hölder norms satisfy:
\[
\left|fg\right|_{\alpha}\leq \left|f\right|_0\left[g\right]_{\alpha}+
\left|f\right|_{\alpha}\left|g\right|_0\leq |f|_{\alpha}|g|_{\alpha}.
\]
In the special case where $|f|\geq 1$,
\[
\left|\frac{g}{f}\right|_{\alpha}\leq |g|_{\alpha}|f|_{\alpha}.
\]
\end{lemma}
\noindent{The proof is elementary and follows from definition of the Hölder norms.}
\newline

\begin{lemma}
\label{lm: Lipscholder}
Let $u\in C^{2,\alpha}(\Omega)$ be a fixed function and consider $f:\Omega\to \mathbb{R}$ a function 
$f(x)=f\left(x,\nabla u\left(x\right)\right)$ satisfying that for every $x, y\in \Omega$
\begin{equation*}
|f(x)-f(y)|\leq \left|h\left(x\right)-h\left(y\right)\right|+K\left|\ \left|\nabla u\right|^m(x)-\left|\nabla u\right|^m(y) \right|
\end{equation*}
for some positive constant $K$ and some $m\geq 2$. Then:
\begin{equation*}
|f|_{\alpha}\leq \left|h\right|_{\alpha}+K|u|_{1, \alpha}^m.
\end{equation*}
Moreover, for any $g\in C^{\alpha}(\Omega)$, then $h\cdot f\in C^{\alpha}(\Omega)$ and:
\begin{equation*}
|g\cdot f|_{\alpha}\leq  (\left|h\right|_{\alpha}+ K |u|^m_{1, \alpha})|g|_{\alpha}.
\end{equation*}

\end{lemma}
\medskip

\begin{proof}
Take $x, y\in \Omega$, $x\neq y$, by the  hypotheses of the lemma,
\[
|f(y)-f(x)|\leq \left|h\left(x\right)-h\left(y\right)\right|+K\left|\ |\nabla u|^m(y)-|\nabla u|^m(x)\right|
\]
Therefore:
\[
[f]_{\alpha}\leq \left[h\right]_{\alpha}+K[|\nabla u|^m]_{\alpha}.
\]
Using this inequality, Lemma \ref{lm: holderprop} and the fact that $|u|_{1, \alpha}=|u|_0+|\nabla u|_{\alpha}$:
\begin{eqnarray*}
|f|_{\alpha}&=&|f|_0+[f]_{\alpha}\\
&\leq& \left|h\right|_0+\left[h\right]_{\alpha}+K|\nabla u|_0^m+K[|\nabla u|^m]_{\alpha}\\
&\leq&\left|h\right|_{\alpha}+K |\nabla u|^m_{\alpha}\\
&\leq& \left|h\right|_{\alpha}+K |u|_{1, \alpha}^m.
\end{eqnarray*}

The second inequality follows from Lemma \ref{lm: holderprop}, since:
\[
|g\cdot f|_{\alpha}\leq |g|_\alpha|f|_{\alpha}. 
\]\end{proof}

\section{General Existence}
\label{Sec: GenExis}

Before we begin with the proofs of our results we introduce a definition. 

\begin{de}
Let $f:\Omega\to \mathbb{R}$ a $C^{\alpha}$-function, $f(x)=f(x,u, \nabla u)$. We say that a sequence $\{u_i\}$ in $C^{2, \alpha}(\Omega)$ is a {\bf Dirichlet iteration sequence relative to $f$,} starting at $\varphi\in C^{\alpha}(\Omega)$, if it is defined recursively through:
\begin{equation}
    \begin{cases}
    \triangle u_{i+1}(x)=f(x, u_i, \nabla u_i) & x\in\Omega,\\
    u_{i+1}(x)=0 & x\in\partial\Omega
    \end{cases}
\end{equation}
with $u_0=\varphi$. 

\end{de}

We have the following fundamental lemma.

\begin{lemma}
\label{lm:unifbound0}
Let $\{u_i\}$ be a Dirichlet iteration sequence relative to a function $F$ starting at 
$u_0=0$.  Suppose that there exists a bounded function $\psi:\left[0, \infty\right)\to \mathbb{R}$, such that $|f\left(x, u,\nabla u\right)|_{\alpha}\leq \psi(|u|_{1, \alpha})$. 
Assume $\psi$ is increasing and that $\Lambda \psi$ has a fixed point.
Then $\{u_i\}$ is uniformly bounded in $C^{2, \alpha}(\Omega)$.
\end{lemma}

\begin{proof}
Let $x_0$ be the fixed point of $\Lambda \psi$. Then notice that if 
$\left|u_i\right|_{2,\alpha}\leq x_0$ then we have that
\[
|u_{i+1}|_{\alpha}\leq \Lambda |f \left(\cdot,u_i, \nabla u_i\right)|_{\alpha}\leq \Lambda \psi(|u_{i}|_{1,\alpha})\leq x_0.
\]
Since $u_0=0$, the statement of the lemma follows by induction.
\end{proof}

\subsection{Proof of Theorem A}

In this section, we shall assume that 
\[
f=h\left(x\right)+F\left(\nabla u\right),
\]
with $F$ satisfying the hypothesis of Theorem \ref{thmA},
and we shall employ the notation
\[
f_j\left(x\right)=h(x)+F\left(\nabla u_j\left(x\right)\right).
\]
and assume that
\begin{equation}
\label{cond: Lips}
|F(x)-F(y)|\leq  K\left|\ | \nabla u|^m(x)-|\nabla u|^m(y) \right|
\end{equation}
with Lipschitz constant $K$, for some $m\geq 2$.

\medskip
The first observation is obtained by means of the classical Schauder estimates and Lemma \ref{lm: Lipscholder}:

\medskip

\begin{prop}
Let $\{u_i\}$ be the sequence in $ C^{2,\alpha}(\Omega)$ of the solutions to the Dirichlet problems:
\begin{equation*}
    \begin{cases}
    \triangle u_{i+1}(x)=f_i(x) & x\in\Omega,\\
    u_{i+1}(x)=0 & x\in\partial\Omega
    \end{cases}
\end{equation*}
with $u_0=0$, where $f_i:=f(\cdot, \nabla u_i)$, for a function satisfying condition (\ref{cond: Lips}). Then the $C^{2, \alpha}$-norm of the $(i+1)$-term satisfies
\begin{equation*}
    |u_{i+1}|_{2, \alpha}\leq \Lambda \left( \left|h\right|_{\alpha}+K|u_{i}|_{1, \alpha}^m\right)
\end{equation*}
where $\Lambda$ is the Schauder constant. 
\end{prop}
\medskip

\begin{prop} 
\label{prop: unifbound}
Let $\{u_i\}$ be a Dirichlet iteration sequence relative to a function $F$, starting at $0$, where $F$ satisfies the Lipschitz condition (\ref{cond: Lips}). Then, for $K$ small enough, $\{u_i\}$ is uniformly bounded in $C^{2, \alpha}(\Omega)$.
\end{prop}
\begin{proof}
The function $\Lambda \psi(t)=\Lambda(\left|h\right|_{\alpha}+Kt^m)$ is increasing on $[0, \infty)$,
and if $\Lambda K$ is small,
then it has a fixed point. 
\end{proof}
\medskip

\begin{rmk}
 It is important for what follows to notice that the smallest
fixed point of $\lambda\psi\left(t\right)$ is of order $\Lambda\left|h\right|_{\alpha}$ as $K\rightarrow 0$. This can easily be 
seen when $m=2$. Indeed, the smallest fixed point in this case is given by
\[
\frac{1-\sqrt{1-4\Lambda^2K\left|h\right|_{\alpha}}}{2\Lambda K}\sim \Lambda \left|h\right|_{\alpha}.
\]
From now on, we shall assume that $K$ is such that the smallest fixed point of $\Lambda \psi$ is smaller
than $2\Lambda \left|h\right|_{\alpha}$.
This fact is important because of the following. Suppose that an unbounded domain $\Omega$
can be exhausted by a family of bounded open sets $\Omega_j$ whose Schauder constants are
uniformly bounded by $\Lambda$. Then, if the iteration process is applied in each of these
bounded sets, we obtain a sequence of solutions that is uniformly bounded in $C^{2,\alpha}$ in compact
subsets of $\Omega$, and thus a subsequence will converge to a solution of the Dirichlet problem
in $\Omega$. We shall apply this reasoning in the slab to obtain an existence result for the
Dirichlet problem. This will make precise in Section \ref{sec: Unbounded}.
\end{rmk}

%
%
%

\begin{lemma}
\label{lm: estimaparts}
Let $\{u_i\}$ be the sequence as before, then for every $\eta>0$ and $i\geq 1$:
\begin{eqnarray*}
&\left|(u_{i+1}\left(x\right)-u_i\left(x\right))(F_i\left(x\right)-F_{i-1}\left(x\right))\right|&\\
&\leq&\\ 
&mC^{m-1}K \left|u_{i+1}\left(x\right)-u_i\left(x\right)\right|
\left|\nabla u_i\left(x\right)-\nabla u_{i-1}\left(x\right)\right|,&
\end{eqnarray*}

\noindent where $C$ is a uniform bound of $\{u_i\}$ in $C^{2, \alpha}(\Omega)$.
\end{lemma}
\medskip

\begin{proof}
Let $w(s)=s^m$ on a closed interval $I$. Observe that, by the Mean Value Theorem, there exists a number $s_1<C_0<s_2$ such that:
\[
w(s_1)-w(s_2)=mC_0^{m-1}(s_1-s_2).
\]
Hence, if $s_1=|\nabla u_i|$ and $s_2=|\nabla u_{i-1}|$, at some fixed point $x$ in $\Omega$, then we may bound
\[
\frac{|\nabla u_i|^m-|\nabla u_{i-1}|^m}{|\nabla u_i|-|\nabla u_i|}\leq mC^{m-1},
\]
where $C$ is the uniform bound of $\{u_i\}$ in $C^{2, \alpha}(\Omega)$.

\medskip
On the other hand, from the $K$-Lipschitz condition on $F$, evaluating in any but fixed $x\in \Omega$ we have the inequality of real numbers:
\begin{eqnarray*}
|(u_{i+1}-u_i)(F_i-F_{i-1})|&\leq& K|(u_{i+1}-u_i)|(|\nabla u_i|^m-|\nabla u_{i-1}|^m)\\
&\leq& mC^{m-1}K |u_{i+1}-u_i|(|\nabla u_i|-|\nabla u_{i-1}|)\\
&\leq& mC^{m-1}K |u_{i+1}-u_i||\nabla u_i-\nabla u_{i-1}|.
\end{eqnarray*}

Then, the lemma follows.\end{proof}
\medskip


%
%

\medskip

\medskip

\begin{prop}
\label{prop: Sobound}
Let $\{u_i\}$ be the Dirichlet iteration sequence relative to a function $F$, which satisfies the Lipschitz condition (\ref{cond: Lips}). Assume that $\{u_i\}$ is unifomly bounded in $C^{2, \alpha}(\Omega)$ by some constant $C$, then for any integer $i>1$:

\[
\left\|u_{i+1}-u_i\right\|_{1, 2}\leq {{mC^{m-1}}K}\kappa
\left\|(u_i-u_{i-1})\right\|_{1, 2},
\]
where
\[
\kappa=\left(\frac{ |\Omega|}{\omega_n}\right)^{1/n} \quad\mbox{or}
\quad \kappa=\frac{\delta\left(\Omega\right)}{\sqrt{2}}.
\]


\end{prop}
\medskip

\begin{proof}
By the integration by parts formula we have:
\[
\int_{\Omega}(u_{i+1}-u_i)\triangle (u_{i+1}-u_i)= \left\|\nabla(u_{i+1}-u_i)\right\|_2^2.
\]
Hence, by Lemma \ref{lm: estimaparts},
Cauchy-Schwartz and the Poincaré inequality (\ref{ineq: Poinc})
or (\ref{ineq: PoincSlab}), we obtain that: 
\begin{eqnarray*}
\left\|\nabla(u_{i+1}-u_i)\right\|_2^2&\leq&mC^{m-1}K\left\|u_{i+1}-u_i\right\|_2
\left\|\nabla(u_{i}-u_{i-1})\right\|_2\\
&\leq&
mC^{m-1}K\kappa\left\|\nabla u_{i+1}-\nabla u_i\right\|_2
\left\|\nabla(u_{i}-u_{i-1})\right\|_2.
\end{eqnarray*}
Therefore:
\[
\left\|\nabla(u_{i+1}-u_i)\right\|_2\leq{{mC^{m-1}}K}\kappa\left\|\nabla(u_{i}-u_{i-1})\right\|_2
\]
which is what we wanted to prove.
%
\end{proof}
\medskip

Observe that the above proposition can be extended to unbounded domains. The proof is the same, with the only change of the use of the slab Poincaré inequality (\ref{ineq: PoincSlab}), and with the consideration of the Schauder estimates in this context, see Section \ref{sec: Unbounded}. 
\medskip

The combined use of Lemma \ref{lm: estimaparts}, and Propositions \ref{prop: unifbound} and \ref{prop: Sobound} leads to an existence and uniqueness result for the Poisson equation with boundary conditions. 
\medskip

\begin{teo}
Let 
\[
f= h\left(x\right)+F\left(\nabla u\right),
\]
where $h$ is a H\"older continuous function. Suppose that for any $u\in C^{2, \alpha}(\Omega)$, and for every $x, y\in \Omega$:
\[
\left|F(\nabla u\left(x\right))-F(\nabla u\left(y\right))\right|\leq K\left|\ | \nabla u|^m(x)-|\nabla u|^m(y) \right|,
\]
for some constant $K>0$, and a real number $m\geq 2$. Then,
for $K$ small enough, the Dirichlet problem:
\begin{equation}
    \begin{cases}
    \triangle u(x)= f & x\in\Omega,\\
    u(x)=0 & x\in\partial\Omega
    \end{cases}
\end{equation}
has a unique solution in $C^{2, \alpha}(\Omega)$ in a neighborhood of $u$,
in $C^{2,\alpha}\left(\Omega\right)$, of radius $\sim \Lambda\left|h\right|_{\alpha}$. 
\end{teo}

\medskip

\begin{proof}
By Proposition \ref{prop: unifbound}, for $K$ small enough, the sequence $\{u_i\}_{i}$ is uniformly bounded above 
in $C^{2,\alpha}\left(\Omega\right)$ by $C\sim \Lambda \left|h\right|_{\alpha}$. To be more precise,
as there is a $K_0$ such that if $K\leq K_0$ the fixed point of $\Lambda \psi$ is smaller than 
$2\Lambda\left|h\right|_{\alpha}$, so impose this requirement on $K$. 
Therefore,
for every $0<\beta<\alpha$,
any subsequence has a convergent subsequence in $C^{2,\beta}(\Omega)$. Next, we shall show that all these 
subsequential limits are the same, and thus the whole original sequence is convergent
in $C^{2,\beta}\left(\Omega\right)$ and from this our theorem follows.

\medskip
Thus, in order to prove that all subsequential limits are the same, notice that
\[
\left\|u_{i+1}-u_i\right\|_{1, 2}\leq mC^{m-1}K \left(\frac{ |\Omega|}{\omega_n}\right)^{1/n}
\left\|(u_i-u_{i-1})\right\|_{1, 2},
\]
and hence if we require that
\[
K < \min \left\{\frac{1}{m C^{m-1}} \left(\frac{\omega_n}{|\Omega|}\right)^{1/n}, K_0\right\},
\]
our claim follows and the theorem is proved. 

\medskip
If we use the slab Poincar\'e inequality, instead of the "volumetric" one, we can estimate
\[
\left\|u_{i+1}-u_i\right\|_{1, 2}\leq \frac{mC^{m-1}K\delta}{\sqrt{2}}
\left\|(u_i-u_{i-1})\right\|_{1, 2},
\]
and this time if we require that
\[
K < \min \left\{\frac{\sqrt{2}}{m C^{m-1} \delta}, K_0\right\},
\]
again the theorem is proved.
\end{proof}
\medskip


\begin{rmk}
The reader must notice that solutions to the previous Dirichlet problem are unique within
the ball in $C^{2,\alpha}\left(\Omega\right)$ of radius given by the
smallest fixed point of the function
\[
\sigma\left(t\right)=\Lambda\left(\left|h\right|_{\alpha}+Kt^m\right).
\]
\end{rmk}

\medskip
\begin{rmk}[Nonhomogenous boundary conditions]\label{rmk:nonhomogenous}
Here we show how the procedure above extends to nonhomogenous boundary conditions. Indeed,
in this case if we impose that $u|_{\partial \Omega}=\varphi$, then we take as $u_0$ the 
solution to the Dirichlet problem
\begin{equation*}
    \begin{cases}
    \triangle u_{0}(x)=h & x\in\Omega,\\
    u_{0}(x)= \varphi& x\in\partial\Omega
    \end{cases}
\end{equation*}
and define for $i=1, 2, 3, \dots$ the Dirichlet iterations as before. In this case, we obtain via
Schauder estimates
\[
\left|u_{i+1}\right|_{2,\alpha}\leq \Lambda \left(\left|h\right|_{\alpha}+
\left|\varphi\right|_{2,\alpha}+K\left|u\right|^m_{1,\alpha}\right),
\]
and thus the fact that the function
\[
\sigma\left(t\right)=\Lambda\left(\left|h\right|_{\alpha}+
\left|\varphi\right|_{2,\alpha}+Kt^m\right)
\]
has a fixed point for $K$ small enough, which in for implies that for $K$ 
small enough the Dirichlet sequence is uniformly bounded in $C^{2,\alpha}$
by a bound $\sim \Lambda\left(\left|h\right|_{\alpha}+
\left|\varphi\right|_{2,\alpha}\right)$. 

\medskip
On the other hand, to show that all convergent subsequences of the Dirichlet iteration
sequence converge towards the same limit, we use the same argument as before. The reader
must notice is that for the difference between two consecutive terms of the
Dirichlet iteration sequence $u_{i+1}-u_{i}=0$ at the boundary, so the proof of this fact
given for homogeneous boundary conditions apply verbatim to the case of nonhomogenous boundary
conditions.
\end{rmk}

\subsection{Proof of Theorem B}
In this section we assume that $f$ has the following form
\[
f\left(x,u,\nabla u\right)=\gamma\left(x\right)g\left(u\right)\left|\nabla u\right|^m+h\left(x\right).
\]

We shall also assume that $\left|g'\left(x\right)\right|\leq x^k$ for all $x$,
and that $\left|g'\left(x\right)\right|\leq g'\left(\left|x\right|\right)$ and that it is increasing.
First we have an estimate
\[
\left|f\right|_{\alpha}\leq \left|h\right|_{\alpha}+\left|\gamma\right|_{\alpha}\left|g\left(u\right)\right|_{\alpha}
\left|\nabla u\right|_{\alpha}^m,
\]
but as 
\[
\left|g\left(u\right)\right|_{\alpha}\leq \left|g'\left(\left|u\right|_0\right)\right|\left|u\right|_{\alpha},
\]
{and $\left|u\right|_0\leq \delta \left|\nabla u\right|_0$, for $\delta$
the slab diameter of $\Omega$}, then we have an estimate
\[
\left|f\right|_{\alpha}\leq \left|h\right|_{\alpha}+g'\left(\delta\left|u\right|_{2,\alpha}\right)
\left|u\right|_{2,\alpha}\left|u\right|_{2,\alpha}^m.
\]
Hence, each solution to the Dirichlet iteration satisfies
\[
\left|u\right|_{2,\alpha}\leq \Lambda \left(\left|h\right|_{\alpha}+\left|\gamma\right|_{\alpha}\delta^{k-1}\left|u\right|^{m+k}_{2,\alpha}
\right).
\]
Thus if either $\left|\gamma\right|_{\alpha}$ is small enough or if $d<1$ if $k$ is large enough, then the function
\[
\Lambda\left(\left|h\right|_{\alpha}+\left|\gamma\right|_{\alpha}\delta^{k-1}t\right)
\]
has a fixed point, say $C=O\left(\Lambda \left|h\right|_{\alpha}\right)$. It follows that the solutions to the 
Dirichlet iterations are uniformly bounded in $C^{2,\alpha}(\Omega)$ by $C$. Therefore, any subsequence
to the Dirichlet iteration has a convergent subsequence; we now need to show that
all the subsequential limits are the same.
Having this in mind, we consider the equation satisfied by $\Delta \left(u_{i+1}-u_i\right)$ just as before,
we obtain an estimate
\begin{eqnarray*}
&\left|\gamma\left(x\right)g\left(u_{i}\right)\left|\nabla u_i\right|^m - \gamma\left(x\right)g\left(u_{i-1}\right)\left|\nabla u_i\right|^m\right|&\\
&\leq&\\
&\left|\gamma\right|\left|g\left(u_i\right)-g\left(u_{i-1}\right)\right|\left|\nabla u_{i}\right|^m
+
\left|\gamma\right|g\left(u_{i-1}\right)\left|\left|\nabla u_i\right|^m-\left|\nabla u_{i-1}\right|^m\right|.
&
\end{eqnarray*}
On the one hand, we estimate, assuming first that $u_{i-1}\left(x\right)\geq 0$,
\begin{eqnarray*}
\left|g\left(u_{i-1}\right)\right|=\left|\int_0^{u_{i-1}}g'\left(t\right)\,dt\right|&\leq& \int_0^{u_{i-1}}\left|g'\left(t\right)\right|\,dt \\
&\leq& \frac{u_{i-1}\left(x\right)^{k+1}}{k+1}.
\end{eqnarray*}
If $u_{i-1}\left(x\right)\leq 0$, we can estimate:
\begin{eqnarray*}
\left|g\left(u_{i-1}\right)\right|&=&\left|\int_0^{u_{i-1}}g'\left(t\right)\,dt\right|\\
&\leq& \int_{u_{i-1}}^0\left|g'\left(t\right)\right|\,dt \leq\int_{u_{i-1}}^0 g'\left(\left|t\right|\right)\,dt\\
&\leq& \int_{u_{i-1}}^0 \left|t\right|^{k}\,dt \leq \frac{\left|u_{i-1}\left(x\right)\right|^{k+1}}{k+1}.
\end{eqnarray*}
On the other hand,
\begin{eqnarray*}
\left|g\left(u_{i}\left(x\right)\right)-g\left(u_{i-1}\left(x\right)\right)\right|&=&
\left|g'\left(\xi\right)\right|\left|u_{i}\left(x\right)-u_{i-1}\left(x\right)\right|\\
&\leq& g'\left(\left|\xi\right|\right)\left|u_{i}\left(x\right)-u_{i-1}\left(x\right)\right|\\
&\leq& g'\left(\max\left\{\left|u_{i-1}\left(x\right)\right|,\left|u_{i}\left(x\right)\right|\right\}\right)
\left|u_{i}\left(x\right)-u_{i-1}\left(x\right)\right|\\
&\leq& C^{k}\left|u_{i}\left(x\right)-u_{i-1}\left(x\right)\right|.
\end{eqnarray*}
In this case then we can bound
\begin{eqnarray*}
&\displaystyle\int_{\Omega} \left|f_{i}-f_{i-1}\right|\left|u_{i+1}-u_{i}\right|\,dx&\\
&\leq&\\
&\left|\gamma\right|_0 \displaystyle
\int_{\Omega}\left(C^{k}\left|u_{i}\left(x\right)-u_{i-1}\left(x\right)\right|\left|\nabla u_i\right|^m
+mC^{m-1}\dfrac{\left|u_{i-1}\right|^{k+1}}{k+1}\left|\nabla \left(u_{i}-u_{i-1}\right)\right|\right)
\times&\\
&\left|u_{i+1}-u_{i}\right|&\\
&\leq&\\
&\left|\gamma\right|_0C^{k+m}\left(\left\|u_{i}-u_{i-1}\right\|_2\left\|u_{i+1}-u_{i}\right\|_2
+\dfrac{m}{k+1}\left\|\nabla \left(u_{i}-u_{i-1}\right)\right\|_2\left\|u_{i+1}-u_{i}\right\|_2\right)&\\
&\leq&\\
&\left|\gamma\right|_0C^{k+m}\left(\left(\dfrac{ |\Omega|}{\omega_n}\right)^{2/n}\left\|\nabla \left(u_{i}-u_{i-1}\right)\right\|_2
\left\|\nabla \left(u_{i+1}-u_{i}\right)\right\|_2 \right.&\\
&\left. \qquad \qquad \qquad \qquad +\left(\dfrac{ |\Omega|}{\omega_n}\right)^{1/n} \dfrac{m}{\left(k+1\right)}\left\|\nabla \left(u_{i}-u_{i-1}\right)\right\|_2
\left\|\nabla\left(u_{i+1}- u_{i}\right)\right\|_2\right).&
\end{eqnarray*}
And hence, we obtain an estimate
\[
\left\|\nabla\left(u_{i+1}-u_i\right)\right\|_{2}\leq \left|\gamma\right|_{0} B C^{m+k}\left(\frac{ |\Omega|}{\omega_n}\right)^{1/n}\left\|\nabla\left(u_{i}-u_{i-1}\right)\right\|_{2},
\]
with
\[
B=\max\left\{\left(\frac{ |\Omega|}{\omega_n}\right)^{2/n}, \left(\frac{ |\Omega|}{\omega_n}\right)^{1/n} \frac{m}{\left(k+1\right)}\right\}.
\]
Thus, if $\left|\gamma\right|_0$ is small enough, there is a solution to the Dirichlet problem in $\Omega$. This completes the proof of Theorem \ref{thmB}.
\newline

{\begin{rmk}
The proof above can be worked out for unbounded domains so that the
bounds do not depend on the volume of $\Omega$, but 
instead on its slab diameter. The proof follows along the same lines, using the slab Poincar\'e inequality
instead of the "volumetric" one. In this case
the value $\left(\left|\Omega\right|/\omega_n\right)^{\frac{1}{n}}$ for the the constant
in the Poincar\'e inequality must be replaced by $\delta/\sqrt{2}$, so
we obtain in the previous proof that the constant $B$ above
can be taken as
\[
B=\max\left\{\frac{\delta^2}{2},\frac{m\delta}{\left(k+1\right)\sqrt{2}}\right\}.
\]
Again, we can conclude that if $\left|\gamma\right|_{\alpha}$ is small enough, then
we have existence, but in this case how small $\left|\gamma\right|_{\alpha}$ must be
is dictated by the slab diameter of $\Omega$ and its Schauder constant. This shall be important
in our attempt to extend this result to unbounded domains.
\end{rmk}}

\section{Existence of solutions for the prescribed mean curvature problem}
\label{sec: pmcp}

The objective
of this section is to
prove an existence result for
prescribed mean curvature problem. We follow closely the scheme used to obtain the previous results to the prescribed mean curvature problem. As announced, we will obtain the Serrin's theorem for bounded domains, as well on one bounded direction domains.
\newline

The prescribed mean curvature problem under consideration in this paper is the following:
\begin{equation}
\begin{cases}
\nabla\cdot \left(\dfrac{\nabla u}{\sqrt{1+\left|\nabla u\right|^2}}\right)=H & \mbox{in} \quad \Omega\\
u=0 & \mbox{on} \quad \partial\Omega
\end{cases}
\end{equation}
for a $C^{2, \alpha}$ function $u:\Omega\to \mathbb{R}$ defined on a Lipschitz domain $\Omega$ in $\mathbb{R}^n$. Where $H=H(x)$. If we expand out the equation, it reads as
\[
\frac{\partial^{\mu}\partial_{\mu} u}{\sqrt{1+\left|\nabla u\right|^2}}
-\frac{2\partial^{\mu} u\partial^{\nu}u \partial_{\mu\nu}u}{\left(1+\left|\nabla u\right|^2\right)^{\frac{3}{2}}}
=nH,
\]
where we are using Einstein's summation convention. The previous expression
can be further rewritten as the elliptic equation:
\[
\Delta u = n\sqrt{1+\left|\nabla u\right|^2}H+
\frac{2\partial^{\mu} u\partial^{\nu}u \partial_{\mu\nu}u}{1+\left|\nabla u\right|^2}
\]
Observe that we are regarding $H$ as a function independent of $u$.  
\newline

From the intermediate value theorem, it follows easily that for $x,y\geq 0$,
\[
\left|\sqrt{1+y}-\sqrt{1+x}\right|\leq \frac{1}{2}\left|x-y\right|.
\]
Therefore, the function $\sqrt{1+|\nabla u\left(x\right)|^2}$ is Lipschitz in $|\nabla u|^2$ with constant $K=1/2$.
\newline

For any fixed $u\in C^{2, \alpha}(\Omega)$, let $F_u=\sqrt{1+|\nabla u|^2}$ and $G_u=2\partial^{\mu} u\partial^{\nu}u \partial_{\mu\nu}u$. For $H$ in $C^{\alpha}(\Omega)$, we propose the following scheme to solve the equation above
\[
\Delta u_{i+1}=\sqrt{1+\left|\nabla u_i\right|^2}H+\frac{2\partial^{\mu} u_i\partial^{\nu}u_i \partial_{\mu\nu}u_i}{{1+\left|\nabla u_i\right|^2}},
\]
with $u_0=0$ and $u_{i+1}|_{\partial \Omega}=0$. That is, $\{u_i\}$ is a Dirichlet iteration sequence relative to $\displaystyle{F_uH+\frac{G_u}{F_u^2}}$. Note that for any $u\in C^{2, \alpha}(\Omega),$ $|F_u|\geq 1$. Observe also that this function does not satisfies the Lipschitz condition (\ref{cond: Lips}).  We will prove that $\{u_i\}$ converges to a classical solution $u$ in $C^{2, \alpha}(\Omega)$.
\medskip

\medskip

\begin{lemma}
\label{lm: GuHolder}
For a fixed $u\in C^{2, \alpha}(\Omega)$ and $H$ as above, $G_u\in C^{\alpha}(\Omega)$ and its Hölder norm reads:
\[
\left|G_u\right|_{\alpha}
\leq 2n^2\left(\left|u\right|_{2, \alpha}\left|u\right|_{1, \alpha}\left|\nabla u\right|_{\alpha}
\right)\leq 2n^2|u|_{2, \alpha}^3
\]

\end{lemma}
\medskip

\begin{proof}





It is straightforward. The function $G_u$ is given by
\[
G_u=\langle \nabla u, \nabla F_u\rangle.
\]
Then it is a sum of $2n^2$ terms of the form $\partial^{\mu}u\partial^{\nu}u\partial_{\mu \nu}u$, whose $C^{\alpha}$-norm is bounded above by:
\begin{eqnarray*}
\left|\partial^{\mu}u\partial^{\nu}u\partial_{\mu \nu}u\right|_{\alpha}&\leq&\left|u\right|_{2, \alpha}\left|u\right|_{1, \alpha}\left|\nabla u\right|_{\alpha}
\end{eqnarray*}
We are using Lemma 1 and the fact that $C^{k+1, \alpha}(\Omega)\subseteq C^{k, \alpha}(\Omega)$.

\end{proof}
\medskip

For short, for each $i\geq 1$ we will write $F_i=F_{u_i}$ and $G_i=G_{u_i}$. 

\begin{lemma}
Let $\{u_i\}$ be the Dirichlet iteration sequence relative to $F\cdot H+\frac{G}{F^2}$, where $F$ and $G$ are as before. Then:
\[
|u_{i+1}|_{2, \alpha}\leq\Lambda\left(1+|u_i|_{1, \alpha}^2\right)\left( |H|_{\alpha}+2n^2|u_i|_{2, \alpha}|u_i|_{1, \alpha}|\nabla u_i|_{\alpha}(1+|u_i|_{1, \alpha}^2)\right)
\]

\end{lemma}
\medskip

\begin{proof}
This follows from Lemmas \ref{lm: holderprop}, \ref{lm: Lipscholder}, \ref{lm: GuHolder} and the Schauder estimates. In fact, in terms of $\Lambda$, the Schauder constant of the Laplacian, we have that
\begin{eqnarray*}
|u_{i+1}|_{2, \alpha}&\leq& \Lambda \left |F_i H+\frac{G_k}{F_i}\right|_{\alpha}\\
&\leq& \Lambda |F_i|_{\alpha} \left( |H|_{\alpha}+ |G_i|_{\alpha} |F_i|_{\alpha}\right)\\
&\leq &\Lambda \left(1+ \frac{|u_i|_{1, \alpha}^2}{2}\right)\left(|H|_{\alpha}+|G_k|_{\alpha}\left(1+ \frac{|u_i|_{1, \alpha}^2}{2}\right)\right)\\
&\leq& \Lambda \left(1+ {|u_i|_{1, \alpha}^2}\right)\left(|H|_{\alpha}+|G_i|_{\alpha}\left(1+{|u_i|_{1, \alpha}^2}\right)\right)\\
&\leq& \Lambda \left(1+|u_i|_{1, \alpha}^2\right)\left( |H|_{\alpha}+2n^2|u_i|_{2, \alpha}|u_i|_{1, \alpha}|\nabla u_i|_{\alpha}(1+|u_i|_{1, \alpha}^2)\right).
\end{eqnarray*}

\end{proof}
\medskip

%

\begin{teo}
Consider the prescribed mean curvature problem,:
\begin{equation}
\begin{cases}
\triangle u = \sqrt{1+|\nabla u|^2}H+\dfrac{2\partial^{\mu}u\partial^{\nu}u\partial_{\mu\nu}u}{1+|\nabla u|^2} & x\in \Omega\\
u=0 &  x\in\partial\Omega,
\end{cases}
\end{equation}
on a smooth domain $\Omega$ which is either bounded or {a slab}. 
{There is an $\varepsilon\left(\Omega\right)>0$ such that if 
$\left|H\right|_{\alpha}<\varepsilon\left(\Omega\right)$, then the prescribed mean curvature problem has 
a solution.}
\end{teo}
\medskip

\begin{proof}
Let $\{u_i\}$ the Dirichlet iteration sequence of the latter lemma. From the inequality
\[
|u_{i+1}|_{2, \alpha}\leq \Lambda\left(1+|u_i|_{1, \alpha}^2\right)\left( |H|_{\alpha}+2n^2|u_i|_{2, \alpha}|u_i|_{1, \alpha}|\nabla u_i|_{\alpha}(1+|u_i|_{1, \alpha}^2)\right)
\]
we obtain that:
\[
|u_{i+1}|_{2, \alpha}\leq \Lambda\left(1+|u_i|_{2, \alpha}^2\right)\left( |H|_{\alpha}+2n^2|u_i|_{2, \alpha}^3(1+|u_i|_{2, \alpha}^2)\right).
\]
\medskip

Then the function $f(x, |H|_{\alpha}, n)=\left(1+x^2\right)\left( |H|_{\alpha}+2n^2x^3(1+x^2)\right)$ is a polynomial in $x$ of degree $5$,   and there
is an $\epsilon>0$ such that
if $\left|H\right|_{\alpha}<\epsilon$ and $n\geq2$ then it has at least one fixed point
which we shall denote by $C\sim \Lambda\left|H\right|_{\alpha}$. Therefore the sequence $\{u_i\}$ is uniformly bounded in $C^{2, \alpha}(\Omega)$ by  $C$.
\newline

Let $v_{i+1}=u_{i+1}-u_i$. Then we can estimate:
\begin{eqnarray*}
\int_{\Omega}v_{i+1}\triangle v_{i+1}=\left\|\nabla v_{i+1}\right\|_2^2&\leq& \left|\int_{\Omega}v_{i+1} \left[(F_i-F_{i-1})H+\frac{G_i}{F_i}-\frac{G_{i-1}}{F_{i-1}} \right] \right|\\
&\leq& \int_{\Omega}|v_{i+1}| (|F_i-F_{i-1}) H|+ \left| \int_{\Omega} v_{i+1}\left( \frac{G_i}{F_i}-\frac{G_{i-1}}{F_{i-1}}\right) \right|\\
&\leq&\int_{\Omega} \frac{\left|H\right|}{2}\left|v_{i+1}\right|\left|\nabla u_{i}+\nabla u_{i-1}\right|
\left|\nabla u_i-\nabla u_{i-1}\right| \\
&&+\left| \int_{\Omega} v_{i+1}\left( \frac{G_i}{F_i}-\frac{G_{i-1}}{F_{i-1}}\right) \right|.
\end{eqnarray*}
In other terms, it can be written as:
\[
\left\|\nabla(v_{i+1})\right\|_2^2\leq{C}\left|H\right|_0\left\|v_{i+1}\right\|_2\left\|\nabla v_i\right\|_2+
\left| \int_{\Omega} v_{i+1}\left( \frac{G_i}{F_i}-\frac{G_{i-1}}{F_{i-1}}\right) \right|.
\]
By the Poincaré inequality, there exists a constant $\kappa>0$ such that:
\[
\left\|\nabla v_{i+1}\right\|_2^2\leq
\kappa{C}\left|H\right|_0\left\|\nabla v_{i+1}\right\|_2\left\|\nabla v_i\right\|_2+
\left| \int_{\Omega} v_{i+1}\left( \frac{G_i}{F_i}-\frac{G_{i-1}}{F_{i-1}}\right) \right|.
\]

On the other hand, we estimate for any but fixed point $x\in \Omega$,
\begin{eqnarray*}
\left|\frac{G_i}{F_i^2}-\frac{G_{i-1}}{F_{i-1}^2} \right|&\leq &|G_i F^2_{i-1}-G_{i-1}F_i^2|\\
&\leq&  |(G_i -G_{i-1})F_{i-1}^2|+|G_{i-1}(F_{i-1}^2- F_{i}^2)|\\
&=&|(G_i-G_{i-1})(1+|\nabla u_{i-1}|^2)|+|G_{i-1}(|\nabla u_{i-1}|^2-|\nabla u_i|^2)|\\
&\leq& \left(1+C^2\right)\left|G_{i}-G_{i-1}\right|+
2n^2C^3\left|\left|\nabla u_{i-1}\right|^2-\left|\nabla u_i\right|^2\right|.
\end{eqnarray*}
Next we bound
\[
\int_{\Omega}v_{i+1}\left|\left(G_i-G_{i-1}\right)\left(1+\left|\nabla u_{i-1}\right|^2\right)\right|.
\]
First notice that we can write
\begin{eqnarray*}
G_{i}-G_{i-1}&=&2\partial^{\mu}\left(u_{i}-u_{i-1}\right)\partial^{\nu} u_i \partial_{\mu\nu}u_i\\
&&+2\partial^{\mu}u_{i-1}\partial^{\nu}\left(u_i-u_{i-1}\right)\partial_{\mu\nu}u_i\\
&&+2\partial^{\mu}u_{i-1}\partial^{\nu} u_{i-1} \partial_{\mu\nu}\left(u_i-u_{i-1}\right)
\end{eqnarray*}
Hence, we have that 
\begin{eqnarray*}
\int_{\Omega} v_{i+1}\left|G_i-G_{i-1}\right| &\leq&2\left[ C^2\left\|v_{i+1}\right\|_2
\left\|\nabla u_{i}-\nabla u_{i-1}\right\|_2\right.\\
&&+C^2\left\|v_{i+1}\right\|
\left\|\nabla u_{i}-\nabla u_{i-1}\right\|_2\\
&&+C^2\left\|\nabla v_{i+1}\right\|_2\left\|\nabla u_{i}-\nabla u_{i-1}\right\|_2\\
&&\left.+2C^2\left\|v_{i+1}\right\|_2\left\|\nabla u_{i}-\nabla u_{i-1}\right\|_2\right],
\end{eqnarray*}
where to estimate
\[
\int_{\Omega} v_{i+1}\partial^{\mu}u_{i-1}\partial^{\nu} u_{i-1} \partial_{\mu\nu}\left(u_i-u_{i-1}\right)
\]
we have used integration by parts and the fact that $v_{i+1}=0$ at the boundary to obtain
\[
-\int_{\Omega} \partial_{\mu}\left(v_{i+1}\partial^{\mu}u_{i-1}\partial^{\nu} u_{i-1}\right) \partial_{\nu}\left(u_i-u_{i-1}\right)
\,
\]
and then the estimate.

\medskip
Thus
\begin{eqnarray*}
\int_{\Omega}v_{i+1} \left|\frac{G_i}{F_i^2}-\frac{G_{i-1}}{F_{i-1}^2} \right|&\leq & 
2\left[\left(1+C^2\right) C^2\left\|v_{i+1}\right\|_2
\left\|\nabla u_{i}-\nabla u_{i-1}\right\|_2 \right.\\
&&+\left(1+C^2\right)C^2\left\|v_{i+1}\right\|_2
\left\|\nabla u_{i}-\nabla u_{i-1}\right\|_2\\
&&+\left(1+C^2\right)C^2\left\|\nabla v_{i+1}\right\|_2\left\|\nabla u_{i}-\nabla u_{i-1}\right\|_2\\
&&\left.+2\left(1+C^2\right)C^2\left\|v_{i+1}\right\|_2\left\|\nabla u_{i}-\nabla u_{i-1}\right\|_2\right]\\
&&+4n^2C^4\left\|v_{i+1}\right\|_2\left\|u_i-u_{i-1}\right\|_2.
\end{eqnarray*}
Therefore, using the Poincar\'e inequality we obtain the estimate

\begin{eqnarray*}
\left| \int_{\Omega} v_{i+1}\left( \frac{G_i}{F_i}-\frac{G_{i-1}}{F_{i-1}}\right) \right|&\leq& 
q\left(\kappa, C\right)\left\|v_{i+1}\right\|_{2}\left\|u_{i}-u_{i-1}\right\|_{2}
\end{eqnarray*}
for some bounded real function $q$ and such that for $\kappa$ fixed tends to 0 as $C\rightarrow 0$,
which in turn tends to 0 as $\left|H\right|_{\alpha}\rightarrow 0$.
\newline

Then we obtain:
\begin{eqnarray*}
\left\|\nabla v_{i+1}\right\|_2&\leq& \sqrt{2}\kappa C\left|H\right|_0\left\|\nabla v_i\right\|_2 + q\left(\kappa, C\right)\left\|u_i-u_{i-1}\right\|_2\\
&\leq& \left(\sqrt{2}\kappa C\left|H\right|_0+{\kappa}q\left(\kappa, C\right)\right)\left\|\nabla v_i\right\|_2,
\end{eqnarray*}
and this last inequality implies that for $\left|H\right|_{\alpha}$ small enough, all the 
subsequences of $\left\{u_i\right\}$ converge towards the same limit. We can take either
$\kappa=\left( \dfrac{|\Omega|}{\omega_n}\right)^{1/n}$ or $\kappa=\delta/\sqrt{2}$,
where
$\delta$ is the slab diameter of $\Omega$. 
This finishes the proof of the theorem.\end{proof}

\section{Unbounded domains: Existence}
\label{sec: Unbounded}
{As we claimed, our existence results extend to some unbounded domains
(we still do not know if the uniqueness statements also extend). We have made some comments on this 
in the 
previous sections. 
In particular, there is a subtle consideration to make with respect to the Schauder estimates.} To show how 
our results can be extended to some unbounded domains, we shall work out the case of the strip
\[
\Pi_{d}=\mathbb{R}\times\left(-d/2,d/2\right),
\]
as we mentioned in the introduction, however, it is not difficult to treat
other domains bounded in one direction even in more dimensions.

\medskip
In the case of a strip, 
with $f$ as in Theorem \ref{thmA} and \ref{thmB}, or in the form
\[
f=n\sqrt{1+\left|\nabla u\right|^2}H+
\frac{2\partial^{\mu} u\partial^{\nu}u \partial_{\mu\nu}u}{1+\left|\nabla u\right|^2},
\]
as is the case of the prescribed mean curvature equation,
we shall show that the Dirichlet problem 
\begin{equation}
\label{eq:unbounded}
\left\{ \begin{array}{l}
    \Delta u = f \quad \mbox{in} \quad \Pi_{d}\\
    u=0 \quad \mbox{on}\quad \partial \Pi_{d},
    \end{array}
    \right.
\end{equation}
has a solution 
for $K$ small enough.
To do so, we first consider the family of Dirichlet problems
\begin{equation}
\label{eq:approximation}
\left\{ \begin{array}{l}
    \Delta u = f \quad \mbox{in} \quad \Pi_{n,d}\\
    u=0 \quad \mbox{on}\quad \partial \Pi_{n,d},
    \end{array}
    \right.
\end{equation}
where $\Pi_{n,d}$ can be described as follows: it is the domain
bounded by the segments $\left[-n,n\right]\times\left\{\pm d/2\right\}$
and joining the points by means of 
translations of curves $\Gamma_{-1}$ and $\Gamma_1$
such that they join
$\left(-n,-d/2\right)$ and $\left(-n,d/2\right)$ and
$\left(n,-d/2\right)$ and $\left(n,d/2\right)$ (we have called 
this curves $\Gamma_n$ and $\Gamma_{-n}$ in Figure 1) respectively such that 
the whole boundary is smooth (see Figure 1).
We have that in each subdomain $\Pi_{n,d}$ the following estimate holds
for $K$ small enough:
\[
\left|u\right|_{2,\alpha}\leq 2\Lambda_n \left|h\right|_{\alpha}.
\]
where $\Lambda_n$ is the Schauder constant of the subdomain $\Pi_{n,d}$. Since each of
the rounding curves are all congruent, by the comments in page 98 of \cite{GilTru1983} 
(right after the proof of Lemma 6.5) there is a $\Lambda$ such that 
$\Lambda_n\leq \Lambda$.

\begin{figure}
\label{fig:slab}
\centering
\begin{tikzpicture}
\draw (-8,0) -- (8,0);
\node at (-5.3, -0.2){$-n$};
\draw[rounded corners=0.3cm, fill=black!3] (-5, 0) rectangle (6.5, -2);
\node at (-5.1, -2.3){$\Gamma_{-n}$};
\draw[rounded corners=0.3cm, pattern=north west lines, pattern color=black!60] (-1.8, 0) rectangle (3.3, -2);
\node at (-1.2, -2.3){$\Gamma_{-(n-1)}$...};
\draw[rounded corners=0.3cm, pattern=dots, pattern color=black!40] (-3.2, 0) rectangle (4.7, -2);
\node at (-3.2, -2.3){$\Gamma_{-(n-2)}$};
\node at (6.7, -0.2){$n$};
\draw[<-, thick] (-6.6, 0) -- (-6.6, -0.8);
\node at (-6.7, -1){$d$};
\draw[->, thick] (-6.6, -1.2) -- (-6.6, -2);
\draw (-8,-2) -- (8,-2);
\node at (5.1, -2.3){$\Gamma_{n-2}$};
\node at (3.7, -2.3){$\Gamma_{n-1}$};
\node at (6.1, -2.3){$\Gamma_{n}$};
\end{tikzpicture}
\caption{An exhaustion of the slab: all the subdomains have their Schauder constants bounded above by a common value $\Lambda$}.
\end{figure}
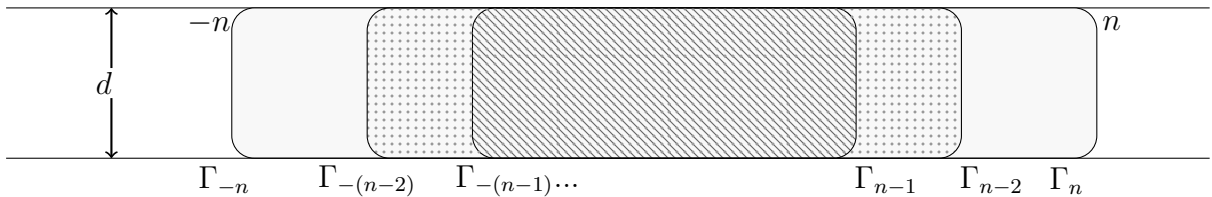

Therefore, the solutions to problem (\ref{eq:approximation}) via Dirichlet iterations are all uniformly bounded 
in $C^{2,\alpha}$ by the smallest fixed point of the function, which is of order 
$\Lambda \left|h\right|_{\alpha}$, for $K$ small. Let then $u_n$ be the solution
to the Dirichlet problem in $\Pi_{n,d}$ via the Dirichlet iterations. Then
for $n\geq N$
they are uniformly bounded in $C^{2,\beta}\left(\Pi_{N,d}\right)$; thus
we can pick a subsequence $u_{n_k}$ that converges in $C^{2,\frac{\beta}{2}}\left(\Pi_{n,d}\right)$
for all $n$, say to $u$, which clearly is a solution to the Dirichlet problem in the whole strip.

	\end{document}